\documentclass[12pt]{article}
\usepackage{amsmath,amssymb,amsthm,threeparttable}
\usepackage{graphicx}
\usepackage{overpic}

\normalsize\textwidth140mm\hoffset=-5mm%
\normalsize\parindent=1em%
\theoremstyle{plain} \newtheorem{theorem}{Theorem}
\theoremstyle{plain} \newtheorem{lemma}{Lemma}[section]
\theoremstyle{remark} 
\theoremstyle{plain} \newtheorem{proposition}[lemma]{Proposition}
\theoremstyle{plain} \newtheorem{corollary}[lemma]{Corollary}
\theoremstyle{plain} 
\theoremstyle{definition} 
\theoremstyle{definition} 
\theoremstyle{definition} \newtheorem{step}{Step}
\newcommand{\rmthree}{\mathrm{I\hspace{-1pt}I\hspace{-1pt}I}}
\newcommand{\rmfour}{\mathrm{I\hspace{-1pt}V}}
\newcommand{\rmfive}{\mathrm{V}}
\newcommand{\rmsix}{\mathrm{V\hspace{-1pt}I}}
\newcommand{\rmseven}{\mathrm{V\hspace{-1pt}I\hspace{-1pt}I}}

\newcommand{\adim}{D}

\newcommand{\ch}{\mathrm{ch}}

\newcommand{\threealg}[3]{\rmthree({#1},{#2},{#3})}
\newcommand{\threeeven}[1]{\rmthree_{\mathrm{e}}({#1},{#1})}
\newcommand{\fourone}[2]{\rmfour_1({#1},{#2})}
\newcommand{\fourtwo}[3]{\rmfour_2({#1},{#2},{#3})}

\newcommand{\foureven}[1]{\rmfour_{\mathrm{e}}({#1},{#1})}
\newcommand{\barfourtwo}{\widehat{\rmfour}_2(2,2,\frac{1}{2})}

\newcommand{\fivethree}[1]{\rmfive_3({#1},{#1})}

\newcommand{\sixthree}[1]{\rmsix_3({#1},{#1})}
\newcommand{\sevenalg}{\rmseven(\frac{4}{3},\frac{4}{3})}

\newcommand{\ifusion}{\mathfrak{V}(4,3)}
\newcommand{\chf}[1]{\ch\,\mathbb{#1}}

\newcommand{\axdsum}[4]{\bigoplus_{#3=#1}^{#2}{#4}a_{#3}}
\newcommand{\haxdsum}[4]{\bigoplus_{#3=#1}^{#2}{#4}\hat{a}_{#3}}
\newcommand{\spanby}[2]{\mathrm{Span}_{#1}\{{#2}\}}
\newcommand{\algenei}[1]{\langle a_0,a_{#1}\rangle_{alg}}

\title{On primitive axial decomposition algebras of Majorana type with degenerate eigenvalues}
\author{Takahiro Yabe\thanks{
Graduate School of Mathematical Sciences, The University of Tokyo
3-8-1, Komaba, Meguro-ku, Tokyo 153-8914, Japan, email: tyabe@ms.u-tokyo.ac.jp}}
\date{}

\begin{document}
\maketitle

\begin{abstract}\noindent
A class of axial decomposition algebras with Miyamoto group generated by two Miyamoto automorphisms and three eigenvalues $0,1$ and $\eta$ is introduced and classified in the case with $\eta\notin\{0,1,\frac{1}{2}\}$.
This class includes specializations of 2-generated axial algebras of Majorana type $(\xi,\eta)$ to the case with $\xi=\eta$.
\end{abstract}

\section{Introduction}
In \cite{i09}, A.\ A.\ Ivanov introduced Majorana algebras by axiomatizing some properties of the 196,884 dimensional Griess algebra called the monstrous Griess algebra.
In \cite{hrs13}, J.\ I.\ Hall et al.\ introduced a class of nonassociative algebras called axial algebras, which contains Majorana algebras and many other algebras.
An axial algebra is a commutative nonassociative algebra generated by axes, a distinguished set of idempotents, subject to a condition called a \textit{fusion rule} given in terms of eigenvalues of axes.
For a fusion rule $\mathcal{F}$ and an algebra $A$, an idempotent $a$ of $A$ is called an $\mathcal{F}$-axis if its action on $A$ is semisimple and the eigenspaces obey the fusion rule $\mathcal{F}$.
An algebra generated by $\mathcal{F}$-axes is called an $\mathcal{F}$-axial algebra.
A Majorana algebra is a $\ifusion$-axial algebra over $\mathbb{R}$ with an associative, symmetric and positive definite bilinear form, where $\ifusion$ is the fusion rule with the set of eigenvalues $\{0,1,\frac{1}{4},\frac{1}{32}\}$ coming from the fusion rules of the Ising model in conformal field theory.
See \cite{hrs13} for details.

In our previous work \cite{y20}, we considered the classification of 2-generated axial algebras with the fusion rule $\mathcal{F}(\xi,\eta)$ as Table \ref{tabfrule}.
\begin{table}[h]
\begin{center}
\begin{tabular}{c|cccc}
$\star$&0&1&$\xi$&$\eta$\\ \hline
0&0&0&$\xi$&$\eta$\\ 
1&0&0&$\xi$&$\eta$\\ 
$\xi$&$\xi$&$\xi$&$\{0,1\}$&$\eta$\\ 
$\eta$&$\eta$&$\eta$&$\eta$&$\{0,1,\xi\}$
\end{tabular}
\caption{The fusion rule $\mathcal{F}(\xi,\eta)$}
\label{tabfrule}
\end{center}
\end{table}
We call an $\mathcal{F}(\xi,\eta)$-axial algebra an axial algebra \textit{of Majorana type $(\xi,\eta)$}.
A Majorana algebra is an axial algebra of Majorana type $(\frac{1}{4},\frac{1}{32})$.
Since the axes are defined by the fusion rule and eigenspaces, in that paper, we assumed that four eigenvalues $0,1,\xi$ and $\eta$ are distinct. 

The classification of \cite{y20} includes some algebras which allow specialization to $\xi=\eta$.
Although such specializations are not axial algebras, they are included in the class  \textit{decomposition algebras} introduced in \cite{dpsv19}, which generalizes axial algebras.

In this paper, we will consider a generalization of $\mathcal{F}(\xi,\eta)$-axial algebras that includes the case with $\xi=\eta$ by using that language.

Unfortunately, the specializations of $\mathcal{F}(\xi,\eta)$-axial algebras to $\xi=\eta$ are not 2-generated in general. 
However, they satisfy a condition which we call \textit{dihedral}.
That is, the algebra $A$ is generated by a sequence of axes $\{a_i\}_{i\in\mathbb{Z}}$ such that the Miyamoto involution of $a_0$ sends $a_i$ to $a_{-i}$ and there exists an automorphism of $A$ which sends $a_i$ to $a_{i+1}$.

The purpose of this paper is to classify dihedral primitive axial decomposition algebras of Majorana type.
When $\xi\neq\eta$, the condition dihedral is equivalent to being 2-generated and having an automorphism called flip, which switches $a_0$ and $a_1$.
So the classification for $\xi\neq\eta$ reduces to the results of our previous work \cite{y20} if $\chf{\mathbb{F}}\neq5$ and to \cite{fm21} and \cite{y20} if $\chf{\mathbb{F}}=5$.
The main result of this paper is the classification of dihedral axial decomposition algebras of Majorana type $(\eta,\eta)$ with $\eta\neq\frac{1}{2}$.

In Section 2, we introduce dihedral axial decomposition algebras of Majorana type $(\eta,\eta)$ and study their fundamental properties.
In Section 3, we give a list of dihedral axial decomposition algebras of Majorana type $(\eta,\eta)$.
In Section 4, we state our main theorem that the dihedral axial decomposition algebras of Majorana type $(\eta,\eta)$ with $\eta\neq\frac{1}{2}$ is isomorphic to one of the algebras listed in Section 3.
We also give an outline of the proof.
Finally, we give the details of proofs in Section 5. 

\section{Preliminaries}
Let $\mathbb{F}$ be a field such that $\chf{\mathbb{F}}\neq2$ and $M$ a commutative nonassociative algebra over $\mathbb{F}$.

Take $I=\{0,1,2,3\}$ as an index set of four elements and let $\Phi$ be a function from $I$ to $\mathbb{F}$ such that $\Phi(0)=0$, $\Phi(1)=1$.
We call $a\in M$ a (primitive) \textit{$\mathcal{F}$-axis (with parameters $\Phi$)} if there exists a decomposition $M=\bigoplus_{i=0}^3M^i(a)$ such that
\begin{itemize}
\item[(1)]For all $i\in\{0,1,2,3\}$ and $x\in M^i(a)$, $xa=\Phi(i)x$,
\item[(2)]$M^1(a)=\mathbb{F} a$,
\item[(3)]$M^0(a)M^i(a)\subset M^i(a)$ for all $i$, $M^2(a)M^2(a)\subset M^0(a)\oplus M^1(a)$, $M^2(a)M^3(a)\subset M^3(a)$ and $M^3(a)M^3(a)\subset\bigoplus_{i=0}^2M^i(a)$.
We call this fusion rule $\mathcal{F}$.
\end{itemize}
We simply call an $\mathcal{F}$-axis an axis.
We call $M$ an \textit{axial $(\mathcal{F},\Phi)$-decomposition algebra} if there exists an $\mathcal{F}$-axis.
(See \cite{dpsv19} for more general axiomatization.)

For an axis $a$, there exists an automorphism $\tau(a)$ of M such that $\tau(a)(x)=x$ if $x\in\bigoplus_{i=0}^2M^i(a)$ and $\tau(a)(x)=-x$ if $x\in M^3(a)$.
This automorphism is called the Miyamoto involution of $a$.

In \cite{r14} or \cite{y20}, the classification of axial $(\mathcal{F},\Phi)$-decomposition algebra generated by two axes are considered in the case when $\Phi$ is an injection.
By using the language of axial decomposition algebra, we generalize such classification for the case with $\Phi(2)=\Phi(3)$.

In fact, the assumption that $M$ is 2-generated is not appropriate for such generalization since the following proposition holds.

\begin{proposition}\sl
Let $\hat{M}$ be a commutative nonassosiative algebra over $R=\mathbb{F}[x_1,\ldots x_n]$, $\xi$ and $\eta$ maps from $\mathbb{F}^n$ to $\mathbb{F}$, $I$ a subset of $\hat{M}$ such that $\hat{M}=\bigoplus_{i\in I}Ri$ and $l,m\in I$.

For $x\in\mathbb{F}^n$, let an algebra $M(x)$ be the $\mathbb{F}$-space  $M=\bigoplus_{i\in I}\mathbb{F}a_i$ with $a_ja_k=\sum_{i\in I}f_i(x)i$ for all $i,j\in I$, where $f_i\in R$ satisfies $jk=\sum_{i\in I}f_ii$.

Assume that  if $0,1,\xi(x)$ and $\eta(x)$ are distinct, then $M(x)$ is an axial $(\mathcal{F},\Phi)$-decomposition algebra generated by two axes $a_l$ and $a_m$.

If $\xi(x)=\eta(x)$ and $|I|>3$, then $M(x)$ is an axial decomposition algebra but not generated by $a_l$ and $a_m$.
\end{proposition}
\begin{proof}
Let $p=a_la_m-\eta(x)(a_l+a_m)$.
Then by a similar calculation as \cite{hrs13} or \cite{y20},
$a_jp=\pi a_0$
for some $\pi\in\mathbb{F}$ and $j\in\{l,m\}$.
Then, since $(a_la_m)(p-\pi a_l)=a_l(a_m(p-\pi a_l))$ by the fusion rule, $pp=\pi p$.
So $\langle a_l,a_m \rangle_{alg}$ is 3-dimensional and thus the proposition holds.
\end{proof}

We call $M$ an \textit{axial decomposition algebra of Majorana type} $(\xi,\eta)$ if there exists a $\Phi$-axis with $\Phi(2)=\xi$ and $\Phi(3)=\eta$.
Moreover, we call $M$ \textit{dihedral} if it is equipped with a sequence of $\Phi$-axes $(a_i)_{i\in\mathbb{Z}}$ such that

\begin{itemize} 
\item[(D1)] $M$ is generated by $(a_i)_{i\in\mathbb{Z}}$.
\item[(D2)] $a_i\mapsto a_{i+1}$ gives an automorphisms of $M$.
\item[(D3)] $a_i\mapsto a_{2j-i}$ gives the Miyamoto involution of $a_j$ for all $j\in\mathbb{Z}$.
\end{itemize}

For the rest of this paper, we assume that $\xi,\eta\notin\{0.1\}$.

\begin{proposition}\sl
Let $M$ be an axial decomposition algebra of Majorana type $(\xi,\eta)$ with $\xi\neq\eta\neq0$.
Then $M$ is dihedral if and only if  $M$ is generated by 2 axes and there exists an automorphism which switches 2 generating axes.
\end{proposition}
\begin{proof}
For two axes $a_0$ and $a_1$ of $M$, there exists a sequence of axes satisfying the condition (D2) and (D3) above if there exists an automorphism switching $a_0$ and $a_1$.

Furthermore, if a sequence of axes $(a_i)_{i\in\mathbb{Z}}$ satisfies (D2) and (D3), then there exists an automorphism switching $a_0$ and $a_1$.

So it suffices to verify that for a sequence $(a_i)_{i\in\mathbb{Z}}$ of axes which satisfies (D2) and (D3), then $a_j\in\langle a_0,a_1\rangle_{alg}$ for all $j$.
This claim is shown if it is proved that $a_{2i-j}\in\langle a_i,a_j\rangle_{alg}$ for all $i,j\in\mathbb{Z}$.

For $i,j\in\mathbb{Z}$, there exists $\lambda\in\mathbb{F}$, $x\in M^2(a_i)$, $y\in M^3(a_i)$ and $z\in M^0(a_i)$ such that
$$a_j=\lambda a_i+x+y+z.$$
Then
$$a_ia_j=\lambda a_i+\xi x+\eta y$$
and
$$a_i(a_ia_j)=\lambda a_i+\xi^2x+\eta^2 y.$$
So 
$$y=\frac{1}{\eta(\eta-\xi)}(a_i(a_ia_j)-\xi a_ia_j+\lambda(\xi-1)a_i)\in\langle a_i,a_j\rangle_{alg}$$
 since $\xi\neq\eta\neq0$.
Since $y=\frac{1}{2}(a_j-a_{2i-j})$ by the definition of $(a_i)_{i\in\mathbb{Z}}$, $a_{2i-j}\in\langle a_i,a_j\rangle_{alg}$.
So the proposition is proved.
\end{proof}

Assume that $M$ is a dihedral axial decomposition algebra of Majorana type $(\eta,\eta)$, $(a_i)_{i\in\mathbb{Z}}$ are generating axes and $\eta\neq\frac{1}{2}$.

Let $f_i$ be an automorphism of $M$ such that $f_i(a_j)=a_{i+j}$, $\tau_i$ the Miyamoto involution of $a_i$, $G=\langle f_1, \tau_0\rangle$ and $\theta=f_1\circ\tau_0\in G$.

We call the dimension of $\spanby{\mathbb{F}}{a_i}_{i\in\mathbb{Z}}$ the \textit{axial dimension} of $M$ and denote it by $\adim$.

It is shown in \cite{y20} that if $\adim$ is finite, then $M$ satisfies exactly one of the following properties:
\begin{itemize}
\item[(1)] $\adim$ is even and there exists a sequence $(\alpha_i)_{0\leq i\leq k}$ of elements of $\mathbb{F}$ such that $\alpha_k\neq0$ and $\sum_{i=1}^k\alpha_i(a_i+a_{-i})+\alpha_0 a_0=0$ with $\adim=2k$
\item[(2)] There exist $k\in\mathbb{Z}$ and a sequence $(\alpha_i)_{1\leq i\leq k}$ of elements of $\mathbb{F}$ such that $\adim=2k$, $\alpha_k\neq0$ and $\sum_{i=1}^k\alpha_i(a_i-a_{-i})=0$.
\item[(3)] There exist $k\in\mathbb{Z}$ and a sequence $(\alpha_i)_{0\leq i\leq k}$ of elements of $\mathbb{F}$ such that $\adim=2k+1$, $\alpha_k\neq0$ and $\sum_{i=0}^k\alpha_i(a_{i+1}+a_{-i})=0$.
\item[(4)] There exist $k\in\mathbb{Z}$ and a sequence $(\alpha_i)_{0\leq i\leq k}$ of elements of $\mathbb{F}$ such that $\adim=2k+1$, $\alpha_k\neq0$ and $\sum_{i=0}^k\alpha_i(a_{i+1}-a_{-i})=0$.
\end{itemize}
We say that $M$ satisfies an \textit{even relation} if (1) or (3) holds and $M$ satisfies an \textit{odd relation} if (2) or (4) holds.

The following lemmas on linear relations among axes will be used later.

\begin{lemma}\label{leminv1}\sl
Let $x,y$ be elements of $M$ such that
\begin{itemize}
\item[\rm{(i)}]$f_2(x)=-\tau_0(x)$, $f_3(x)=x$ and $x+f_1(x)+f_2(x)=0$,
\item[\rm{(ii)}]$f_1(y)=-y$ and $\tau_0(y)=y$.
\end{itemize}
If there exist $\alpha_i\in\mathbb{F}$ for $0\leq i\leq k$ such that $x+y+\sum_{i=0}^k\alpha_i(a_{i+1}-a_{-i})=0$, the following statements hold:
\begin{itemize}
\item[\rm{(1)}]$\alpha_k(a_{k+2}-a_{-k-2})+\sum_{i=1}^{k+1}(\alpha_{i+1}+2\alpha_i+2\alpha_{i-1}+\alpha_{i-2})(a_{i}-a_{-i})=0$, where $\alpha_{k+1}=\alpha_{k+2}=0$ and $\alpha_{-1}=-\alpha_{0}$.
Moreover, in the case when $\alpha_k\neq0$, $\adim\leq2k+4$ and if $\adim=2k+2$, odd relation holds.
\item[\rm{(2)}]If $x=0$, then $\alpha_k(a_{k+1}-a_{-k-1})+\sum_{i=1}^k(\alpha_i+\alpha_{i-1})(a_{i}-a_{-i})=0$.
Moreover, in the case when $\alpha_k\neq0$, $\adim\leq2k+2$ and if $\adim=2k+2$, odd relation holds.
\end{itemize}
\end{lemma}
\begin{proof}
Since
\begin{eqnarray*}
0&=&x+y+\sum_{i=0}^k\alpha_i(a_{i+1}-a_{-i})-\tau_0(x+y+\sum_{i=0}^k\alpha_i(a_{i+1}-a_{-i}))\\
&=&-f_1(x)+\alpha_k(a_{k+1}-a_{-k-1})+\sum_{i=1}^k(\alpha_i+\alpha_{i-1})(a_{i}-a_{-i}),
\end{eqnarray*}
(2) holds if $x=0$, and
\begin{eqnarray*}
0&=&-f_1(x)+\alpha_k(a_{k+1}-a_{-k-1})+\sum_{i=1}^k(\alpha_i+\alpha_{i-1})(a_{i}-a_{-i})\\
&&+f_1(-f_1(x)+\alpha_k(a_{k+1}-a_{-k-1})+\sum_{i=1}^k(\alpha_i+\alpha_{i-1})(a_{i}-a_{-i}))\\
&&+f_{-1}(-f_1(x)+\alpha_k(a_{k+1}-a_{-k-1})+\sum_{i=1}^k(\alpha_i+\alpha_{i-1})(a_{i}-a_{-i}))\\
&=&\alpha_k(a_{k+2}-a_{-k-2})+\sum_{i=1}^{k+1}(\alpha_{i+1}+2\alpha_i+2\alpha_{i-1}+\alpha_{i-2})(a_{i}-a_{-i}),
\end{eqnarray*}
implies (1).
\end{proof}

\begin{lemma}\label{leminv}\sl
Let $x,y,z$ be elements of $M$ such that
\begin{itemize}
\item[\rm{(i)}] $g(x)=x$ for all $g\in G$,
\item[\rm{(ii)}] $f_3(y)=y$, $\tau_0\circ f_1(y)=f_2(y)$ and $\tau_0(y)=y$,
\item[\rm{(iii)}] $\tau_0(z)=z$ and $\tau_0\circ f_1(z)=f_1(z)$,
\item[\rm{(iv)}] $x+y+z+\sum_{i=1}^k\alpha_i(a_i+a_{-i})+\alpha_0a_0=0$ for some $\alpha_i\in\mathbb{F}$.
\end{itemize}
Then the following statements hold.
\begin{itemize}
\item[\rm{(1)}]$\alpha_k(a_{k+2}-a_{-k-2})+\sum_{i=1}^{k+1}(\alpha_{i-2}+\alpha_{i-1}-\alpha_{i+1}-\alpha_{i+2})(a_i-a_{i})=0$ where $a_i=0$ for $i>k$ and $\alpha_{-1}=\alpha_1$.
Thus in the case when $\alpha_k\neq0$, $\adim\leq2k+4$ and if $\adim=2k+4$, then an odd relation holds.
\item[\rm{(2)}]If $y=0$, then $\alpha_k(a_{k+1}-a_{-k-1})+\sum_{i=1}^{k}(\alpha_{i-1}-\alpha_{i+1})(a_i-a_{-i})=0$.
Thus in the case when $\alpha_k\neq0$, $\adim\leq2k+2$ and if $\adim=2k+2$, then an odd relation holds.
\item[\rm{(3)}]If $y=z=0$, then $\alpha_k(a_{k+1}-a_{-k})+\sum_{i=1}^{k-1}(\alpha_i-\alpha_{i+1})(a_{i+1}-a_{-i})=0$.
Thus in the case when $\alpha_k\neq0$, $\adim\leq2k+1$ and if $\adim=2k+1$, then an odd relation holds.
\end{itemize}
\end{lemma}
\begin{proof}
Set $\alpha_i=0$ if $i\geq k+1$.
Since $f_1(x)=x$,
\begin{eqnarray*}
0&=&f_1(x+y+z+\sum_{i=1}^k\alpha_i(a_i+a_{-i}))-(x+y+z+\sum_{i=1}^k\alpha_i(a_i+a_{-i}))\\
&=&f_1(y)-y+f_1(z)-z+\alpha_k(a_{k+1}-a_{-k})+\sum_{i=1}^{k-1}(\alpha_i-\alpha_{i+1})(a_{i+1}-a_{-i}).
\end{eqnarray*}
Thus (3) holds if $y=z=0$

By Lemma \ref{leminv1}, (1) and (2) hold.
\end{proof}

Let $p_{i,j}=a_ja_{i+j}-\eta(a_i+a_{i+j})$ for $i,j\in\mathbb{Z}$.
Then following two properties holds by the results of \cite{y20}.
$$a_0p_{i,0}=((1-\eta)\lambda_i-\eta)a_0$$
and
$$p_{i,0}-(\lambda_i-\eta)a_0+\frac{\eta}{2}(a_i+a_{-i})\in M^2(a_0).$$
Let $x\in M$.
Moreover, for a $G$-invariant element $x$ of $M$, if $xa_0=\pi a_0$ for some $\pi\in\mathbb{F}$, then $xp_{i,j}=\pi p_{i,j}$ for all $i,j\in\mathbb{Z}$.
This property holds by the results of \cite{y20}.

\begin{proposition}\label{propmulti}\sl

The following properties hold.
\begin{itemize}
\item[\rm{(1)}]$a_0p_{2,1}=\frac{(2\eta-1)(4\lambda_1-3\eta)}{2\eta}(2p_1+\eta(a_1+a_{-1}))+\mu a_0$ for some $\mu\in\mathbb{F}$.
\item[\rm{(2)}]For some $\nu\in\mathbb{F}$,
\begin{eqnarray*}
a_0p_{3,1}&=&\frac{\eta}{2}(p_{3,1}-p_{3,-1})-\frac{(2\eta-1)(2\lambda_1-\eta)}{2\eta}(2p_{2,0}+\eta(a_2+a_{-2}))\\
&&-\frac{\mu-\eta\lambda_2+2\eta^2}{\eta}(2p_1+\eta(a_1+a_{-1}))+\frac{\nu}{2}a_0.
\end{eqnarray*}
\item[\rm{(3)}]For some $\rho\in\mathbb{F}$, 
\begin{eqnarray*}
p_{2,0}p_{2,1}&=&\frac{(2\eta-1)(4\lambda_1-3\eta)}{4}(2p_{3,0}+p_{3,1}+p_{3,-1})\\
&&+(\mu+\frac{(2\eta-1)(4\lambda_1-3\eta)(2(2\eta-1)\lambda_1-4\eta^2+\eta)}{2\eta^2})p_{2,0}\\
&&+\frac{(2\eta-1)(4\lambda_1-3\eta)(4(2\eta-1)\lambda_1-\eta\lambda_2-5\eta^2+3\eta)}{\eta^2}p_1\\
&&+\frac{(2\eta-1)^2(4\lambda_1-3\eta)(3\lambda_1-2\eta)}{2\eta}(a_2+a_{-2})\\
&&+\frac{(2\eta-1)(4\lambda_1-3\eta)(2\mu-\eta\lambda_2+2\eta^2)}{2\eta}(a_1+a_{-1})+\rho a_0.
\end{eqnarray*}
\end{itemize}
\end{proposition}
\begin{proof}
\begin{eqnarray*}
M^0(a_0)+\mathbb{F}a_0&\ni&(p_1-(\lambda_1-\eta)a_0+\frac{\eta}{2}(a_1+a_{-1}))^2\\
&&\in\mathbb{F}p_1+\mathbb{F}a_0+\frac{\eta^2}{2}p_{2,1}-\frac{\eta(2\eta-1)(4\eta-3\lambda_1)}{4}(a_1+a_{-1}).
\end{eqnarray*}
Thus (1) holds.

Since
\begin{eqnarray*}
M^0(a_0)+\mathbb{F}a_0\ni&(&p_{2,0}-(\lambda_i-\eta)a_0+\frac{\eta}{2}(a_2+a_{-2}))(p_1-(\lambda_i-\eta)a_0+\frac{\eta}{2}(a_1+a_{-1}))\\
&\in& \mathrm{Span}_{\mathbb{F}}\{a_0,p_1,p_{2,0}\}+\frac{\eta^2}{4}(p_{3,1}+p_{3,-1})\\
&&+\frac{\eta(2\eta-1)(2\lambda_1-\eta)}{4}(a_2+a_{-2})+\frac{\eta(\mu-\eta\lambda_2+2\eta^2)}{2}(a_1+a_{-1}),
\end{eqnarray*}
\begin{eqnarray*}
a_0(p_{3,1}+p_{3,-1})&=&-\frac{(2\eta-1)(2\lambda_1-\eta)}{\eta}(2p_{2,0}+\eta(a_2+a_{-2}))\\
&&-\frac{2(\mu-\eta\lambda_2+2\eta^2)}{\eta}(2p_1+\eta(a_1+a_{-1}))+\nu a_0
\end{eqnarray*}
for some $\nu\in\mathbb{F}$.
Furthermore, $a_0(p_{3,1}-p_{3,-1})=\eta(p_{3,1}-p_{3,-1})$. 
Hence (2) holds.
Set 
$$x=(p_{2,0}-((1-\eta)\lambda_2-\eta)a_0)(p_{2,1}-\frac{(2\eta-1)(4\lambda-3\eta)}{2\eta}(a_1+a_{-1})-(\mu-(2\eta-1)(4\lambda_1-3\eta))a_0)$$
and
$$y=(p_{2,0}-(\lambda_2-\eta)a_0+\frac{\eta}{2}(a_1+a_{-1}))(p_{2,1}-\frac{(2\eta-1)(4\lambda-3\eta)}{2\eta}(a_1+a_{-1})-(\mu-(2\eta-1)(4\lambda_1-3\eta))a_0)$$
Then $y=\frac{1}{\eta}a_0(y-x)$ and hence (3) holds.
\end{proof}

\begin{corollary}\label{corlambda2}\sl
If $\lambda_1=\frac{3\eta}{4}$, then $p_{2,1}=p_{2,0}$ or $\mu=0$.
\end{corollary}
\begin{proof}
If $\lambda_1=\frac{3\eta}{4}$, then $p_{2,1}a_0=\mu a_0$.
Since $p_{2,1}+p_{2,0}$ is $G$-invariant, $(p_{2,1}+p_{2,0})p_{i,j}=(\mu+(1-\eta)\lambda_2-\eta)p_{i,j}$ for all $i,j\in\mathbb{Z}$.
Thus $p_{2,1}p_{2,0}=\mu p_{2,0}$ and then $\mu(p_{2,1}-p_{2,0})=0$.
So the claim holds.
\end{proof}

\section{The list}
\subsection{Specialization of known algebras}
There exist some algebras with $\xi\neq\eta$ which admit the specialization to $\xi=\eta$.
Here is the list of such algebras;
\begin{itemize}
\item primitive axial algebras of Jordan type,
\item $\threealg{\eta}{\eta}{0}$ and $\threealg{\frac{-1}{3}}{\frac{-1}{3}}{0}^{\times}$
\item $\fourone{\frac{1}{4}}{\frac{1}{4}}$,
\item $\fourtwo{2}{2}{\frac{1}{2}}$,
\item $\fourtwo{\xi}{\frac{1-\xi^2}{2}}{\frac{-1}{\xi+1}}$ with $\xi^2+2\xi-1=0$,
\item $\rmfive_1(\frac{-1}{3},\frac{-1}{3})$,
\item $\rmsix_2(\frac{4}{9},\frac{4}{9})$.
\end{itemize}

For the definition of these algebras, see \cite{hrs15} for algebras of Jordan type and \cite{y20} for others.

\subsection{Algebra $\threeeven{\eta}$ and its quotient}
For $\eta\in\mathbb{F}\setminus\{0,1\}$, let $\threeeven{\eta}$ be the $\mathbb{F}$-space $\mathbb{F}\hat{p}_1\oplus\haxdsum{-1}{1}{i}{\mathbb{F}}$ with multiplication given by
\begin{itemize}
\item $x\hat{p}_1=\frac{-\eta(3\eta+1)}{4}x$ for all $x\in\threeeven{\eta}$,
\item $\hat{a}_i\hat{a}_i=\hat{a}_i$ for all $i$,
\item $\hat{a}_i\hat{a}_{i+1}=\hat{p}_1+\eta(\hat{a}_i+\hat{a}_{i+1})$ for $i=-1,0$,
\item $\hat{a}_1\hat{a}_{-1}=(2\eta-1)\hat{p}_1+\eta(\hat{a}_1+\hat{a}_{-1})$.
\end{itemize}
Then this is a dihedral axial decomposition algebra of Majorana type $(\eta,\eta)$.
When $\eta=\frac{-1}{3}$,  the algebra admits a quotient $\threeeven{\frac{-1}{3}}^{\times}=\threeeven{\frac{-1}{3}}/\mathbb{F}\hat{p}_1$.

\subsection{Algebra $\foureven{\eta}$ and its quotient}

For $\eta\in\mathbb{F}\setminus\{0,1,-1\}$, let $\foureven{\eta}$ be the $\mathbb{F}$-space $\mathbb{F}\hat{p}_1\oplus\haxdsum{-1}{2}{i}{\mathbb{F}}$ with multiplication given by
\begin{itemize} 
\item $x\hat{p}_1=\frac{-\eta(3\eta+1)}{4}x$ for all $x\in\foureven{\eta}$
\item $\hat{a}_i\hat{a}_i=\hat{a}_i$ for all $i$,
\item $\hat{a_i}\hat{a}_i=\hat{p}_1+\eta(\hat{a}_o+\hat{a}_i)$ For $i=-1,0,1$
\item $\hat{a}_i\hat{a}_{i+2}=\frac{2\eta}{\eta+1}\eta\hat{p}_1+\eta(\hat{a}_i+\hat{a}_{i+2})$ for $i=-1,0$
\item $\hat{a}_{-1}\hat{a}_2=\frac{5\eta^2-1}{(\eta+1)^2}\hat{p}_1+\eta(\hat{a}_{-1}+\hat{a}_2)$
\end{itemize}
Then this is a dihedral axial decomposition algebra of Majorana type $(\eta,\eta)$.
When $\eta=\frac{-1}{3}$, the algebra admits a quotient $\foureven{\frac{-1}{3}}^{\times}=\foureven{\frac{-1}{3}}/\mathbb{F}\hat{p}_1$.

\subsection{Algebra $\barfourtwo$}

Let $\barfourtwo$ be the $\mathbb{F}$-space $\mathbb{F}\hat{p}_1\oplus\mathbb{F}\hat{p}_{2,0}\oplus\mathbb{F}\hat{p}_{2,1}\oplus\axdsum{-1}{2}{i}{\mathbb{F}}$ with multiplication given by
\begin{itemize} 
\item $x\hat{p}_1=-3x$ for all $x\in M$,
\item $\hat{a}_{i}\hat{p}_{2,j}=-3a_{i}$ if $i-j\in2\mathbb{Z}$,
\item $\hat{a}_{i}\hat{p}_{2,j}=-3p_1-3(\hat{a}_{i-1}+\hat{a}_{i+1})-6a_{i}$ if $i-j\notin2\mathbb{Z}$,
\item $\hat{p}_{2,0}\hat{p}_{2,1}=9(\hat{p}_1+\hat{a}_0+\hat{a}_1+\hat{a}_{-1}+\hat{a}_2)$,
\item $\hat{a}_i\hat{a}_i=\hat{a}_i$ for all $i$,
\item $\hat{a}_i\hat{a}_j=\hat{p}_1+\eta(\hat{a}_i+\hat{a}_j$ if $|i-j|=1$ or 3,
\item $\hat{a}_i\hat{a}_{i+2}=\hat{p}_{2,|i|}+\eta(\hat{a}_i+\hat{a}_{i+2})$ for $i=0,-1$.
\end{itemize}
Then this is a dihedral axial decomposition algebra of Majorana type $(2,2)$ and a quotient
\begin{eqnarray*}
\barfourtwo/(\mathbb{F}&&(p_{2,0}+p_1+2(a_2+a_0)+a_1+a_{-1})\\
&&\oplus\mathbb{F}(p_{2,1}+p_1+a_2+a_0+2(a_1+a_{-1}))
\end{eqnarray*}
is isomorphic to $\fourtwo{2}{2}{\frac{1}{2}}$.

\subsection{Algebra $\fivethree{\eta}$}

Let $\fivethree{\eta}$ be the $\mathbb{F}$-space $M=\haxdsum{-2}{2}{i}{\mathbb{F}}$ with multiplication given by
\begin{itemize}
\item $\hat{a}_i\hat{a}_i=\hat{a}_i$ for all $i$
\item $\hat{a}_i\hat{a}_j=\frac{-\eta}{4}(a_0+\hat{a}_1+\hat{a}_{-1}+\hat{a}_{2}+\hat{a}_{-2})+\eta(\hat{a}_i+\hat{a}_j)$ if $i\neq j$.
\end{itemize}
Then this is a dihedral axial decomposition algebra of Majorana type $(\eta,\eta)$.
When $\eta=\frac{-1}{3}$, the algebra admits a quotient $\fivethree{\frac{-1}{3}}/\mathbb{F}(\hat{a}_2+\hat{a}_{-2}+\hat{a}_1+\hat{a}_{-1}+\hat{a}_0)$ and this quotient is isomorphic to $\foureven{\frac{-1}{3}}^{\times}$.

\subsection{Algebra $\sixthree{\eta}$}

For $\eta\in\mathbb{F}\setminus\{0,1\}$, let $\sixthree{\eta}$ be the $\mathbb{F}$-space $\haxdsum{-2}{3}{i}{\mathbb{F}}\oplus\mathbb{F}p_1$ with multiplication given by
\begin{itemize}
\item $\hat{a}_i\hat{a}_i=\hat{a}_i$,
\item $\hat{a}_i\hat{a}_{i+1}=\hat{p}_1+\eta(a_{i}+a_{i+1})$, 
\item $\hat{a}_i\hat{a}_{i+3}=0$,
\item $\hat{a}_i\hat{a}_{i+2}=\frac{\eta}{2}(\hat{a}_i+\hat{a}_{i+2}-\hat{a}_{i-2})$ where $\hat{a}_i=\hat{a}_{i+6}$,
\item $x\hat{p}_1=\frac{-\eta^2}{2}x$ for all $x\in M$
\end{itemize}
Then $\sixthree{\eta}$ is a dihedral axial decomposition algebra of Majorana type $(\eta,\eta)$ if $\eta=-1\pm\sqrt{2}$.

\subsection{Algebra $\sevenalg$}

Let $\sevenalg$ be the $\mathbb{F}$-space $\mathbb{F}p_1\oplus\axdsum{-3}{3}{i}{\mathbb{F}}$ with multiplication given by
\begin{itemize}
\item $x\hat{p}_1=\frac{-5}{3}x$ for all $x\in\sevenalg$,
\item $\hat{a}_i\hat{a}_i=\hat{a}_i$,
\item $\hat{a}_i\hat{a}_{j}=\hat{p}_1+\frac{4}{3}(\hat{a}_i+\hat{a}_{j})$ if $|i-j|=1$, 2 or 4
\item $\hat{a}_i\hat{a}_{i+3}=\frac{4}{3}(\hat{a}_i+\hat{a}_{i+3})+\hat{p}_{3,i}$ where
$$\hat{p}_{3,0}=\frac{-2}{3}(\hat{a}_3-\hat{a}_{-3})+\frac{1}{3}(\hat{a}_2+\hat{a}_{-2}-\hat{a}_1-\hat{a}_{-1}-\hat{a}_0),$$
$\hat{p}_{3,1}=\frac{5}{3}(\hat{a}_3-\hat{a}_{-2})+\hat{p}_{3,0}$, $\hat{p}_{3,-1}=\frac{5}{3}(\hat{a}_{-3}-\hat{a}_{2})+\hat{p}_{3,0}$ and $\hat{p}_{3,i+3}=\hat{p}_{3,i}$.
\item $\hat{a}_3\hat{a}_{-2}=\hat{p}_1-\hat{p}_{3,1}+\frac{4}{3}\hat{a}_3-\frac{1}{3}\hat{a}_{-2}$ and $\hat{a}_2\hat{a}_{-3}=\hat{p}_1-\hat{p}_{3,-1}+\frac{4}{3}\hat{a}_{-3}-\frac{1}{3}\hat{a}_2.$
\item $\hat{a}_3\hat{a}_{-3}=\hat{p}_{3,0}+\frac{4}{3}(\hat{a}_3+\hat{a}_{-3})$.
\end{itemize}
Then $\sevenalg$ is a dihedral axial decomposition algebra of Majorana type $(\frac{4}{3},\frac{4}{3})$.
The algebra admits a quotient $\sevenalg^{\times}=\sevenalg/\mathbb{F}\hat{p}_1$ if $\chf{\mathbb{F}}=5$.

\section{Main  result}
Our main result is the following.

\begin{theorem}\sl
A dihedral axial decomposition algebra of Majorana type $(\eta,\eta)$ with $\eta\neq0,1,\frac{1}{2}$ is isomorphic to one of the algebras listed in Section 3.
\end{theorem}

Let us now show the outline of the proof.
The details will be given in the next section.

Assume that $\eta\in\mathbb{F}\setminus\{0,1,\frac{1}{2}\}$ and $M$ is a dihedral axial decomposition algebras of Majorana type $(\eta,\eta)$ generated by $\{a_i\}_{i\in\mathbb{Z}}$ not of Jordan type.
We will prove the theorem in the following steps.

\begin{step}\label{stadimleq}
$\adim\leq8$ and in the case with $\adim=8$, $M$ satisfies an odd relation.
Moreover, if $p_{3,1}=p_{3,0}$, $\adim\leq6$ and in the case with $\adim=6$, $M$ satisfies an odd relation.
\end{step}

\begin{step}\label{stadim3}
If $\adim\leq3$, then one of the following holds:
\begin{itemize}
\item[(1)] $M$ is isomorphic to $\threealg{\eta}{\eta}{0}$ or $\threealg{\frac{-1}{3}}{\frac{-1}{3}}{0}^{\times}$.
\item[(2)] $M$ is isomorphic to $\threeeven{\eta}$ or $\threeeven{\frac{-1}{3}}^{\times}$.
\end{itemize}
\end{step}

\begin{step}\label{stadim4}
If $\adim=4$, then one of the following holds:
\begin{itemize}
\item[(1)] $M$ is isomorphic to $\foureven{\eta}$ or $\foureven{\frac{-1}{3}}^{\times}$.
\item[(2)] $M$ is isomorphic to $\fourone{\frac{1}{4}}{\frac{1}{4}}$.
\item[(3)] $M$ is isomorphic to $\barfourtwo$ or $\fourtwo{2}{2}{\frac{1}{2}}$.
\item[(4)] $M$ is isomorphic to $\fourtwo{\xi}{\frac{1-\xi^2}{2}}{\frac{-1}{\xi+1}}$ with $\xi^2+2\xi-1=0$.
\end{itemize}
\end{step}

\begin{step}\label{stadim5}
If $\adim=5$, then one of the following holds:
\begin{itemize}
\item[(1)] $M$ is isomorphic to $\rmfive_1(\frac{-1}{3},\frac{-1}{3})$
\item[(2)] $M$ is isomorphic to $\fivethree{\eta}$
\end{itemize}
\end{step}

\begin{step}\sl\label{stadim6}
If $\adim=6$, then one of the following holds:
\begin{itemize}
\item[(1)] $M$ is isomorphic to $\sixthree{\eta}$ with $\eta^2+2\eta-1=0$.
\item[(2)] $M$ is isomorphic to $\rmsix_2(\frac{4}{9},\frac{4}{9})$.
\end{itemize}
\end{step}

\begin{step}\label{stadim7}
If $\adim=7$, then $M$ is isomorphic to $\sevenalg$.
\end{step}

\begin{step}\label{stadim8}
$\adim\neq8$.
\end{step}

\section{Proof of the main theorem}

\subsection{Proof of Step \ref{stadimleq}}

Since $(p_1-(\lambda_1-\eta)a_0+\frac{\eta}{2}(a_1+a_{-1}))(p_{3,0}-((1-\eta)\lambda_3-\eta))\in M^2(a_0)$,
\begin{eqnarray*}
0&=&\alpha_1p_1+\alpha_2(p_{3,0}+p_{3,1}+p_{3,-1})+\alpha_3(p_{2,0}+p_{2,1})+\alpha_4p_{3,0}+\alpha_5p_{2,0}\\
&&+\frac{\eta^2(2\eta-1)(2\lambda_1-\eta)}{4}(a_2+a_{-2})+\beta_1(a_1+a_{-1})+\beta_0a_0
\end{eqnarray*}
for some $\alpha_i,\beta_i\in\mathbb{F}$.
If $\lambda_1\neq\frac{\eta}{2}$, then by Lemma \ref{leminv} (1), $\adim\leq8$ and odd relation holds when $\adim=8$.
In the case when $p_{3,0}=p_{3,1}$, $g(p_{3,0})=p_{3,0}$ for all $g\in G$.
Thus by Lemma \ref{leminv} (2), the claim holds in this case.

Assume $\lambda_1=\frac{\eta}{2}$, then
\begin{eqnarray*}
0&=&-p_{2,0}p_{2,1}+\frac{(2\eta-1)(4\lambda_1-3\eta)}{4}(p_{3,0}+p_{3,1}+p_{3,-1})\\
&&+\frac{(2\eta-1)(4\lambda_1-3\eta)(4(2\eta-1)\lambda_1-\eta\lambda_2-5\eta^2+3\eta)}{\eta^2}p_1\\
&&+\frac{(2\eta-1)(4\lambda_1-3\eta)}{4}p_{3,0}\\
&&+(\mu+\frac{(2\eta-1)(4\lambda_1-3\eta)(2(2\eta-1)\lambda_1-4\eta^2+\eta)}{2\eta^2})p_{2,0}\\
&&+\frac{(2\eta-1)^2(4\lambda_1-3\eta)(3\lambda_1-2\eta)}{2\eta}(a_2+a_{-2})\\
&&+\frac{(2\eta-1)(4\lambda_1-3\eta)(2\mu-\eta\lambda_2+2\eta^2)}{2\eta}(a_1+a_{-1})+\rho a_0.
\end{eqnarray*}
Thus By lemma \ref{leminv}, the claim holds.

\subsection{Proof of Step \ref{stadim3}}
Assume $\adim\leq3$.
By the argument as the case with $\xi\neq\eta$, $\adim=3$. 

Assume that $M$ satisfies an even relation.
Then there exists $\alpha\in\mathbb{F}$ such that
$$a_2+a_{-1}+\alpha(a_1+a_0)=0.$$
Thus
\begin{eqnarray*}
0&=&a_0(a_2+a_{-1}+\alpha(a_1+a_0))\\
&=&p_{2,0}+(\alpha+1)p_1+(2\eta+\alpha)a_0.
\end{eqnarray*}
By $\tau_1$-invariance, $(2\eta+\alpha)(a_2-a_0)=0$.
Since $\adim=3$, $\alpha=-2\eta$ and thus $p_{2,0}=p_{2,1}=-(\alpha+1)p_1$.

Then
\begin{eqnarray*}
0&=&a_0(p_{2,1}-p_{2,0})\\
&=&\frac{(2\eta-1)(4\lambda-3\eta)}{2\eta}(2p_1+\eta(a_1+a_{-1}))+(\mu-(1-\eta)\lambda_2+\eta)a_0
\end{eqnarray*}
Thus $\lambda_1=\frac{3\eta}{4}$ or $p_1=\frac{-\eta}{2}(a_1+a_{-1})+\beta a_0$ for some $\beta\in\mathbb{F}$.
By Lemma \ref{leminv} (3), odd relation holds if $\lambda\neq\frac{3\eta}{4}$.
So $\lambda_1=\frac{3\eta}{4}$ and then all multiplications are determined.
The quotient which did not appear yet is only $\threeeven{\frac{-1}{3}}^{\times}$.

Assume that $M$ satisfies an odd relation.
Then there exists $\alpha\in\mathbb{F}$ such that
$$a_2-a_{-1}+\alpha(a_1-a_0)=0.$$
Thus
\begin{eqnarray*}
0&=&a_0(a_2-a_{-1}+\alpha(a_1-a_0))\\
&=&p_{2,0}+(\alpha-1)p_1+\alpha(2\eta-1)a_0.
\end{eqnarray*}
Then by a similar argument above, $\alpha=0$ and then $\lambda_1=\frac{3\eta}{4}$ or $p_1=-\frac{\eta}{2}(a_1+a_{-1}+a_0)$.
Then $M$ is isomorphic to $\threealg{\eta}{\eta}{0}$ or of Jordan type $\eta$.
Both cases are known in the case when $\xi\neq\eta$.

\subsection{Proof of Step \ref{stadim4}}
Assume that $\adim=4$.

Assume that $M$ satisfies an even relation.
Then there exists $\alpha,\beta\in\mathbb{F}$ such that
$$a_2+a_{-2}+\alpha(a_1+a_{-1})+\beta a_0=0.$$
Thus
\begin{eqnarray*}
0&=&a_0(a_2+a_{-2}+\alpha(a_1+a_{-1})+\beta a_0)\\
&=&2p_{2,0}+2\alpha p_1+((2\alpha-\beta+2)\eta+\beta)a_0.
\end{eqnarray*}
By the $\tau_1$-invariance, $(2\alpha-\beta+2)\eta+\beta=0$ and thus $p_{2,0}=-2\alpha p_1$.
Hence by a similar argument as the case when $\adim=3$, $\lambda_1=\frac{3\eta}{4}$.

Then
\begin{eqnarray*}
a_2a_{-1}&=&f_1(a_1a_{-2})\\
&=&f_1(a_1(-a_2-\alpha(a_1+a_{-1})-\beta a_0))\\
&=&-\alpha p_{2,0}-(\beta+1)r_1+\eta(a_2+a_{-1})-((\beta+2)\eta+\alpha)a_2
\end{eqnarray*}
By the $\theta$-invariance, $(\beta+2)\eta+\alpha=0$.
Hence $\alpha=\beta=\frac{-2\eta}{\eta+1}$ and then all multiplications are determined.
Then the quotient which does not change the axial dimension is only $\foureven{\frac{-1}{3}}^{\times}$.

Assume that $M$ satisfies an odd relation.
Then there exists $\alpha\in\mathbb{F}$ such that 
$$a_2-a_{-2}+\alpha(a_1+a_{-1})=0.$$
Then
\begin{eqnarray*}
a_2a_{-1}&=&f_1(a_1a_{-2})\\
&=&f_1(a_1(a_2+\alpha(a_1-a_{-1}))\\
&=&p_1-\alpha p_{2,0}+\eta(a_2+a_{-1})-(2\eta-1)a_2.
\end{eqnarray*}
By the $\theta$-invariance, 
$$\alpha(p_{2,1}-p_{2,0})-\alpha(2\eta-1)(a_2-a_{-1})=0.$$
Thus by Lemma \ref{leminv1}, $\alpha=0$ or 1.
In the case when $\alpha=1$, the argument in \cite{y20} works and then $\eta=2=\frac{1}{2}$.
So this case is not appropriate.

So we may assume that $\alpha=0$.
In this case, $\langle a_0,a_2\rangle_{alg}$ is of Jordan type $\eta$.

Assume that $a_0a_2=0$.
Since
\begin{eqnarray*}
M^2(a_0)\ni&& a_2(p_1-(\lambda_1-\eta)a_0+\frac{\eta}{2}(a_1+a_{-1}))\\
&&\in M^2(a_0)+\mathbb{F}a_0+(1-\eta)(\lambda_1-\eta)a_2,
\end{eqnarray*}
$\lambda_1=\eta$.
Thus since $p_{2,1}=-\eta(a_1+a_{-1})$, $\eta=\frac{1}{4}$. Then $M$ is isomorphic to $\fourone{\frac{1}{4}}{\frac{1}{4}}$.

Assume that $a_0a_2\neq0$.
Then $\lambda_2=\frac{\eta}{2}$.
Since $p_{3,1}=p_1$,
$(\lambda_1,\mu)=(\frac{\eta}{2},\frac{-3\eta^2}{2})$ or $p_{2,0}=\delta(p_1+\frac{\eta}{2}(a_1+a_{-1})-\eta(a_0+a_2)$ for some $\delta\in\mathbb{F}$.

Assume that $(\lambda_1,\mu)=(\frac{\eta}{2},\frac{-3\eta^2}{2})$.
Then by the $\theta$-invariance of $p_{2,0}p_{2,1}$, $\eta=2$ or $p_{2,1}=p_{2,0}+\frac{2\eta-1}{2}(a_2-a_{-1}-a_1+a_0)$.
If $\eta=2$, all multiplication is determined and the fusion rule holds.
If $\eta\neq2$, by considering $a_0p_{2,1}$, $p_{2,0}=-\frac{\eta}{2}a_2-a_0$ and thus $\eta=2$.
Then the quotient with $\adim=4$ is only $\fourtwo{2}{2}{\frac{1}{2}}$.

Assume that $p_{2,0}=\delta(p_1+\frac{\eta}{2}(a_1+a_{-1})-\eta(a_0+a_2)$ for some $\delta\in\mathbb{F}$.
Then $a_2+\delta\eta a_0-\frac{\delta}{2}(a_1+a_{-1})=0$.
Then the same argument as the case with $\xi\neq\eta$ works and $M$ is isomorphic to $\fourtwo{\xi}{\eta}{\mu}$ for some $\xi,\eta,\mu\in\mathbb{F}$.
So $(\xi,\eta,\mu)=(2,2,\frac{1}{2})$$\eta=\frac{1-\xi^2}{2}$, $\mu=\frac{-1}{\xi+1}$ and $\xi^2+2\xi-1=0$.

\subsection{Proof of Step \ref{stadim5}}

Assume that $\adim=5$.

Assume that $M$ satisfies an even relation.
Then there exist $\alpha,\beta\in\mathbb{F}$ such that
$$a_3+a_{-2}+\alpha(a_2+a_{-1})+\beta(a_1+a_0)=0.$$

Then 
\begin{eqnarray*}
a_{-1}a_2&=&f_{-1}(a_0a_3)\\
&=&-(\alpha+1)p_{2,1}-(\alpha+\beta)p_1+\eta(a_{-1}+a_2)-(2(\alpha+1)\eta+\beta)a_{-1}.
\end{eqnarray*}
By the $\theta$-invariance, 
$$(\alpha+1)(p_{2,1}-p_{2,0})=(2(\alpha+1)\eta+\beta)(a_2-a_{-1}).$$
So by Lemma \ref{leminv1} (2), $2(\alpha+1)\eta+\beta=0$ and thus $(\alpha+1)(p_{2,1}-p_{2,0})=0$.
Moreover, $p_{3,0}=-(\alpha+1)p_{2,0}-(\alpha+\beta)p_1=p_{3,1}=p_{3,-1}$.

Since
\begin{eqnarray*}
a_{-2}a_2&=&f_1(a_{-1}a_3)\\
&=&-\alpha p_{3,0}-\beta p_{2,0}-(\beta+1)p_1+\eta(p_{-2}+p_2)-(2(\beta+1)\eta+\alpha)a_{-2},
\end{eqnarray*}
by Lemma \ref{leminv} (3), $2(\beta+1)\eta+\alpha=0$.
Thus $\alpha=\beta=\frac{-2\eta}{2\eta+1}\neq0$ and $p_{2,1}=p_{2,0}$

Since $a_0p_{2,1}=a_0p_{2,0}$, $\lambda_1=\frac{3\eta}{4}$
Since
\begin{eqnarray*}
M^0(a_0)\ni&&(p_{2,0}-(\lambda_2-\eta)a_0+\frac{\eta}{2}(a_2+a_{-2}))^2\\
&&\in M^0(a_0)-\frac{\eta(2\eta-1)(4\lambda_2-3\eta)}{4}(a_2+a_{-2}),
\end{eqnarray*}
$(4\lambda_2-3)(2p_{2,0}+\eta(a_2+a_{-2}))\in\mathbb{F}a_0$.
By Lemma \ref{leminv} (3), $\lambda_2=\frac{3\eta}{4}$.

So
\begin{eqnarray*}
M^0(a_0)+\mathbb{F}a_0\ni&&(p_1+\frac{\eta}{4}a_0+\frac{\eta}{2}(a_1+a_{-1}))(p_{2,0}+\frac{\eta}{4}a_0+\frac{\eta}{2}(a_2+a_{-2}))\\
&&\in M^0(a_0)+\mathbb{F}a_0+\frac{\eta^2(2\eta-1)}{8}(a_2+a_{-2}+a_1+a_{-1}).
\end{eqnarray*}
Then $2p_{2,0}+2p_1+\eta(a_2+a_{-2}+a_1+a_{-1})\in\mathbb{F}a_0$.
By Lemma \ref{leminv} (3), $M$ satisfies an odd relation in this case and thus this case is not appropriate.

Assume that $M$ satisfies an odd relation.
Then there exist $\alpha, \beta\in\mathbb{F}$ such that
$$a_3-a_{-2}+\alpha(a_2-a_{-1})+\beta(a_1-a_0)=0.$$

Since
\begin{eqnarray*}
a_{-1}a_2&=&f_1(a_0a_3)\\
&=&(1-\alpha)p_{2,1}+(\alpha-\beta)p_1+\eta(a_{-1}+a_2)-(2\eta-1)\beta a_{-1},
\end{eqnarray*}
$(\alpha-1)(p_{2,1}-p_{2,0})=(2\eta-1)\beta(a_2-a_{-1})$ by the $\theta$-invariance.
By the $\tau_0$-invariance, $\beta=0$ and thus $(\alpha-1)(p_{2,1}-p_{2,0})=0$ and $p_{3,0}=(1-\alpha)p_{2,0}+(\alpha-\beta)p_1=p_{3,1}=p_{3,-1}$.

Since
\begin{eqnarray*}
a_{-2}a_2&=&f_1(a_{-1}a_3)\\
&=&-\alpha p_{3,0}+p_1+\eta(a_{-2}+a_2)-(2\eta-1)\alpha a_{-2},
\end{eqnarray*}
$\alpha=0$ by the $\tau_0$-invariance.
Hence $p_{2,1}=p_{2,0}$ and thus $\lambda_1=\frac{3\eta}{4}$.

Then by the same calculation above, $\lambda_2=\lambda_1$ and $p_{2,0}=-p_1-\frac{\eta}{2}(a_2+a_{-2}+a_1+a_{-1}+a_0)$.
Since $\frac{-\eta(3\eta+1)}{4}p_1=p_1p_{2,0}=\frac{-\eta(3\eta+1)}{4}p_{2,0}$, $\eta=\frac{-1}{3}$ or $p_{2,0}=p_1$.

If $\eta=\frac{-1}{3}$, then $M$ is isomorphic to $\rmfive_1(\frac{-1}{3},\frac{-1}{3})$.

If $p_{2,0}=p_1$, then $a_ia_j=-\frac{\eta}{4}(a_2+a_{-2}+a_1+a_{-1}+a_0)+\eta(a_i+a_j)$ for all $i\neq j$ and $M=\bigoplus_{i=-2}^2\mathbb{F}a_i$.
Then the fusion rule holds.

\subsection{Proof of Step \ref{stadim6}}

Assume that $\adim=6$

Assume that $M$ satisfies an even relation.
Then there exists $\alpha,\beta,\gamma\in\mathbb{F}$ such that
$$a_3+a_{-3}+\alpha(a_2+a_{-2})+\beta(a_1+a_{-1})+\gamma a_0=0.$$

Since
\begin{eqnarray*}
a_2a_{-2}&=&f_1(a_1a_{-3})\\
&=&-\alpha p_{3,-1}-(\beta+1)p_{2,0}-(\alpha+\gamma)p_1+\eta(a_2+a_{-2})-((2\alpha+\gamma+2)\eta+\beta)a_2,
\end{eqnarray*}
$\alpha(p_{3,1}-p_{3,-1})=((2\alpha+\gamma+2)\eta+\beta)(a_2-a_{-2})$ by the $\tau_0$-invariance.
Then $((2\alpha+\gamma+2)\eta+\beta)(a_3-a_{-3}+a_2-a_{-2})=0$.
So $(2\alpha+\gamma+2)\eta+\beta=0$ and $\alpha(p_{3,1}-p_{3,-1})=0$.
Since $M$ satisfies an odd relation if $p_{3,1}=p_{3,0}$ and $\adim=6$, $\alpha=0$.

Since
\begin{eqnarray*}
a_3a_{-2}&=&f_1(a_2a_{-3})\\
&=&-\beta p_{3,0}-\gamma p_{2,1}-(\beta+1)p_1+\eta(a_3+a_{-2})-(2\beta+\gamma+2)\eta a_3,
\end{eqnarray*}
$\beta(p_{3,1}-p_{3,0})-\gamma(p_{2,1}-p_{2,0})-(2\beta+\gamma+2)\eta(a_3+a_{-2})=0$ by the $\theta$-invariance.
Then, by Lemma \ref{leminv} (1), 
\begin{eqnarray*}
0&=&(2\beta+\gamma+2)\eta(a_4-a_{-4}+2(a_3-a_{-3})+2(a_2-a_{-2})+a_1-a_{-1})\\
&=&(2\beta+\gamma+2)\eta(2(a_3-a_{-3}+(3-\beta)(a_2-a_{-2})+(1-\gamma)(a_1-a_{-1})).
\end{eqnarray*}
Hence $2\beta+\gamma+2=0$.
Then $(\beta,\gamma)=(0,-2)$ since $p_{3,1}\neq p_{3,0}$.

Then $a_3a_{-3}=2p_{3,0}+\eta(a_3+a_{-3})+(2\eta-1)a_3$.
But this can not be $\tau_0$-invariant and thus this is not appropriate.

Assume that $M$ satisfies an odd relation.
Then there exist $\alpha,\beta\in\mathbb{F}$ such that
$$a_3-a_{-3}+\alpha(a_2-a_{-2})+\beta(a_1-a_{-1})=0.$$

Since 
\begin{eqnarray*}
a_2a_{-2}&=&f_1(a_1a_{-3})\\
&=&-\alpha p_{3,-1}+(1-\beta)p_{2,0}+\alpha p_1+\eta(a_2+a_{-2})-(2\eta-1)\beta a_2,
\end{eqnarray*}
$\alpha(p_{3,1}-p_{3,-1})=(2\eta-1)\beta(a_2-a_{-2})$ by the $\tau_0$-invariance.
Then $(2\eta-1)\beta(a_3-a_{-3}+a_2-a_{-2}+a_1-a_{-1})=0$ and hence $\beta=0$ or $\alpha=\beta=1$.

If $\alpha=\beta=1$, then $a_4-a_{-4}=a_2-a_{-2}$. 
By considering $\langle a_2,a_{-2} \rangle_{alg}$, this case is not apropriate.   
So $\beta=0$ and then $\alpha(p_{3,1}-p_{3,-1})=0$.

Since 
\begin{eqnarray*}
a_3a_{-2}&=&f_1(a_2a_{-3})\\
&=&\alpha^2p_{3,0}-\alpha p_{2,1}+(-\alpha^2+1)p_1+\eta(a_3+a_{-2})-(2\eta-1)\alpha a_3,
\end{eqnarray*}
$\alpha(p_{2,1}-p_{2,0})+(2\eta-1)\alpha(a_3-_{-2})=0$.
By Lemma \ref{leminv1} (2), $\alpha=0$ or 1.

If $\alpha=1$, then $p_{3,1}=p_{3,0}$ and $p_{2,1}=p_{2,0}-(2\eta-1)(a_3-a_{-2})$.
So
\begin{eqnarray*}
p_{3,0}\in\mathbb{F}a_0+p_{2,0}-\eta(a_3-a_{-2})-\frac{4\lambda_1-3\eta}{2\eta}(2p_1+\eta(a_1+a_{-1}))
\end{eqnarray*}
and by the $\theta$-invariance, this can not hold.

Thus we may assume that $a_3=a_{-3}$.
So $a_0a_3=0$ or $\lambda_3=\frac{\eta}{2}$ by the classification of axial algebras of Jordan type.

Assume that $a_0a_3=0$.
Then
\begin{eqnarray*}
a_3(p_1-(\lambda_1-\eta)a_0+\frac{\eta}{2}(a_1+a_{-1}))=&&\eta p_{2,1}+(1-\eta)(\lambda_1-\eta)a_3+\frac{\eta^2}{2}(a_1+a_{-1})\\
&&\in M^2(a_0).
\end{eqnarray*}
Thus there exist $\gamma,\delta\in\mathbb{F}$ such that
$$p_{2,1}=\delta p_1+\frac{(\eta-1)(\lambda_1-\eta)}{\eta}a_3+\frac{(2\eta-1)(4\lambda_1-3\eta)}{2\eta}(a_1+a_{-1})+\gamma a_0.$$
By the invariance, $\frac{(2\eta-1)(4\lambda_1-3\eta)}{2\eta}=\frac{(\eta-1)(\lambda_1-\eta)}{\eta}$ and $\gamma=0$.

By considering $p_{2,0}p_{2,1}$, $a_0(a_2+a_{-2}+a_0)=\in\mathbb{F}a_0$.
Thus $p_{2,0}=\frac{-\eta}{2}(a_0+a_{-2}+a_2)$.
Then $\eta=-1\pm\sqrt{2}$ and the fusion rule holds.

Hence we may assume $\lambda_3=\frac{\eta}{2}$.

Since
\begin{eqnarray*}
M^0(a_0)\oplus\mathbb{F}a_0\ni(p_{3,0}&&+\frac{\eta}{2}a_0+\eta a_3)(p_1-(\lambda_1-\eta)a_0+\frac{\eta}{2}(a_1+a_{-1}))\\
\in M^0&&(a_0)\oplus\mathbb{F}a_0+\frac{-\eta(2\eta-1)(4\lambda_1-3\eta)}{2}a_3\\
+&&(\frac{-\eta(\mu-\eta\lambda_2+2\eta^2)}{2}-\frac{\eta(2\eta-1)(4\lambda-3\eta)}{4})(a_2+a_{-2})\\
+&&(\frac{\eta\nu}{2}-\frac{\eta(2\eta-1)(2\lambda_1-\eta)}{4}+\frac{3\eta^3}{4}\\
&&-\frac{(\eta^2-(2\eta-1)(2\lambda-\eta))(2\eta-1)(4\lambda_1-3\eta)}{2\eta}-\frac{\eta(\mu-\eta\lambda_2+2\eta^2)}{2})(a_1+a_{-1}),
\end{eqnarray*}
one of the following statements holds:
\begin{itemize}
\item[(1)]$\lambda_1=\frac{3\eta}{4}$, $\mu=\eta\lambda_2-2\eta^2$ and $\nu=\frac{\eta(2\eta-1)}{4}$.
\item[(2)]$\lambda_1=\frac{3\eta}{4}$, $\mu\neq\eta\lambda_2-\eta^2$ and there exist $\delta,\epsilon\in\mathbb{F}$ such that
$$p_{2,0}=-\frac{\eta}{2}(a_2+a_{-2})+\delta(p_1+\eta(a_1+a_{-1}))+\epsilon a_0.$$
\item[(3)]$\lambda_1\neq\frac{3\eta}{4}$ and there exist $\gamma,\delta,\epsilon\in\mathbb{F}$ such that
$$p_{3,0}=-\eta a_3+\delta(p_{2,0}+\frac{\eta}{2}(a_2+a_{-2}))+\epsilon(p_1+\frac{\eta}{2}(a_1+a_{-1}))+\gamma a_0.$$
\end{itemize}

Assume $\lambda_1=\frac{3\eta}{4}$.

If $\algenei{2}$ is of Jordan type, then $p_{2,0}=\frac{-\eta}{2}(a_2+a_{-2}+a_0)$.
Since $a_0p_{2,1}=\mu a_0$, $p_{3,0}=-2p_1-\eta a_3-\eta(a_1+a_{-1})-(\frac{3\eta^2}{2}+\mu)a_0$.
However, by considering the $\theta$-invariance of $p_{3,-1}$, $\adim\leq5$ and thus contradicts the assumption.

Thus we may assume $\lambda_2=\frac{3\eta}{4}$.
If $\mu\neq\eta\lambda_2-2\eta^2$, then $p_{2,0}=\frac{-\eta}{2}(a_2+a_{-2}+a_0)$ by the invariance.
So we may assume that $\mu=\eta\lambda_2-2\eta^2=\frac{-5\eta^2}{4}$.
By Corollary \ref{corlambda2}, $p_{2,0}=p_{2,1}$ or $\chf{\mathbb{F}}=5$.
If $p_{2,0}=p_{2,1}$, then $3\eta^2+\eta=5\eta^2$ and then $\eta=\frac{1}{2}$.

So $\chf{\mathbb{F}}=5$. 
Then we can show that $p_{2,0}p_{3,0}=\frac{-\eta^2(2\eta-1)}{4}a_0$ by a similar calculation as Proposition \ref{propmulti} (3).
Thus 
$$\frac{\eta(2\eta-1)}{4}p_{3,0}=(p_{2,1}+p_{2,0})p_{3,0}=\frac{-\eta^2(2\eta-1)}{4}(a_0+a_{3}).$$
Then $a_0a_3=0$ and this contradicts the assumption.

So we may assume that $\lambda_1\neq\frac{3\eta}{4}$ and there exist $\gamma,\delta,\epsilon\in\mathbb{F}$ such that
$$p_{3,0}=-\eta a_3+\delta(p_{2,0}+\frac{\eta}{2}(a_2+a_{-2}))+\epsilon(p_1+\frac{\eta}{2}(a_1+a_{-1}))+\gamma a_0.$$

By the $f_3$-invariance,
$$\delta(p_{2,1}-p_{2,0})=-(\gamma+\eta)(a_3-a_0)+\frac{\eta(\delta-\epsilon)}{2}(a_2+a_{-2}-a_1-a_{-1}).$$

If $\delta=0$, then $\gamma=-\eta$ and $\epsilon=0$.
Thus $a_0a_3=0$ and this contradicts the assumption.
So we may assume that $\delta\neq0$.

If $\algenei{2}$ is of Jordan type, then $\delta=0$ by considering $a_0p_{2,1}$, so $\lambda_2=\frac{3\eta}{4}$.

Since $p_{2,0}\neq p_{2,1}$, in this case,
$$a_3=\frac{-\delta}{\gamma+\eta}(p_{2,1}-p_{2,0})+a_{2}+a_{-2}-a_{1}-a_{-1}+a_0.$$

Since 
\begin{eqnarray*}
M^2(a_0)\ni(p_1&&-(\lambda_1-\eta)a_0+\frac{\eta}{2}(a_1+a_{-1}))\\
&&(p_{2,1}-\frac{(2\eta-1)(4\lambda_1-3\eta)}{2\eta}(a_1+a_{-1})+(\mu-(2\eta-1)(4\lambda_1-3\eta))a_0)\\
\in&&M^2(a_0)+\frac{-2(5\eta-3)\lambda_1+6\eta^2-5\eta}{4}p_{2,1}\\
-&&(\frac{(2\eta-1)(4\lambda_1-3\eta)((1-\eta)\lambda_1-\eta)}{2\eta}\\
&&+\frac{(2\eta-1)(2\eta+1)(4\lambda-3\eta)}{4}+\frac{\eta^3(3\eta+1)}{8})(a_1+a_{-1})\\
+&&(\lambda_1-\eta)(-\mu+(2\eta-1)(4\lambda-3\eta))a_0,
\end{eqnarray*}
$2(5\eta-3)\lambda_1-6\eta^2+5\eta=0$ or $p_{2,1}\in\mathbb{F}p_1+\mathbb{F}(a_1+a_{-1})+\mathbb{F}a_0$.
If $p_{2,1}\in\mathbb{F}p_1+\mathbb{F}(a_1+a_{-1})+\mathbb{F}a_0$, then $p_{2,1}\in\mathbb{F}p_1$ by the invariance, so $\lambda_1=\frac{3\eta}{4}$.
Thus $2(5\eta-3)\lambda_1-6\eta^2+5\eta=0$.

Then, by considering the coefficient of $a_1+a_{-1}$ and $a_0$, $\mu=(2\eta-1)(4\lambda-3\eta)$, $\eta=\frac{4}{9}$ or $\frac{-1}{3}$ and $\chf{\mathbb{F}}\neq3$.

If $\chf{\mathbb{F}}\neq7$, then $(\eta,\lambda_1)=(\frac{4}{9},\frac{2}{3})$ or $(\frac{-1}{3},\frac{-1}{4})$ since $2(5\eta-3)\lambda_1-6\eta^2+5\eta=0$.
However, $\lambda_1=\frac{3\eta}{4}$ if $\eta=\frac{-1}{3}$ and thus $\eta=\frac{4}{9}$.

If $\chf{\mathbb{F}}=7$, $\frac{4}{9}=\frac{-1}{3}=2$.
Then by considering $a_0(p_{3,1}+p_{3,-1})$ and $a_0p_{3,0}$, $\lambda_1=\frac{2}{3}$.

So $(\eta,\lambda_1)=(\frac{4}{9},\frac{2}{3})$.
Then $M$ is isomorphic to $\rmsix_2(\frac{4}{9},\frac{4}{9})$. 

\subsection{Proof of Step \ref{stadim7}}

Assume that $\adim=7$.
Then $p_{3,1}\neq p_{3,0}$.

Assume that $M$ satisfies an even relation.
Then there exist $\alpha,\beta,\gamma\in\mathbb{F}$ such that
$$a_4+a_{-3}+\alpha(a_3+a_{-2})+\beta(a_2+a_{-1})+\gamma(a_1+a_0)=0.$$

Since 
\begin{eqnarray*}
a_2a_{-2}=&&f_{-2}(a_0a_4)\\
=&&-(\alpha+1)p_{3,1}-(\alpha+\beta)p_{2,0}-(\beta+\gamma)p_1+\eta(a_2+a_{-2})-(2(\alpha+\beta+1)\eta+\gamma)a_{-2},
\end{eqnarray*}
$(\alpha+1)(p_{3,1}-p_{3,-1})=(2(\alpha+\beta+1)\eta+\gamma)(a_2-a_{-2})$.
Then by a similar calculation as Lemma \ref{leminv1}, $2(\alpha+\beta+1)\eta+\gamma=0$.
Since $p_{3,1}\neq p_{3,0}$, $\alpha=-1$.

Since 
\begin{eqnarray*}
a_{-2}a_3&=&f_{-1}(a_{-1}a_4)\\
&=&-\beta p_{3,1}+(\beta-\gamma)p_{2,0}-(\gamma-1)p_1+\eta(a_{-2}+a_3)-(2\gamma\eta+\beta)a_{-2},
\end{eqnarray*}
$\beta(p_{3,1}-p_{3,0})+(\beta-\gamma)(p_{2,1}-p_{2,0})=(2\gamma\eta+\beta)(a_3-a_{-2})$.
By Lemma \ref{leminv1} (1), $2\gamma\eta+\beta=0$ and since $p_{3,1}\neq p_{3,0}$, $\beta=0$.
Then $\gamma=0$

Then
\begin{eqnarray*}
a_3a_{-3}&&=f_{-1}(a_4a_{-2})\\
&&=\eta(a_3+a_{-3})-(2\eta-1)a_{-3}.
\end{eqnarray*}
By $\tau_0$-invariance, this can not be true.

So we may assume $M$ satisfies an odd relation.
Then there exist $\alpha,\beta,\gamma\in\mathbb{F}$ such that
$$a_4-a_{-3}+\alpha(a_3-a_{-2})+\beta(a_2-a_{-1})+\gamma(a_1-a_0)=0.$$

Since 
\begin{eqnarray*}
a_2a_{-2}=&&f_{-2}(a_0a_4)\\
=&&(1-\alpha)p_{3,1}+(\alpha-\beta)p_{2,0}+(\beta-\gamma)p_1+\eta(a_2+a_{-2})-(2\eta-1)\gamma a_{-2},
\end{eqnarray*}
$(\alpha-1)(p_{3,1}-p_{3,-1})=(2\eta-1)\gamma(a_2-a_{-2})$.
Then by similar calculation as Lemma \ref{leminv1}, $\gamma=0$.
Since $p_{3,1}\neq p_{3,0}$, $\alpha=1$.

Since 
\begin{eqnarray*}
a_{-2}a_3&=&f_{-1}(a_{-1}a_4)\\
&=&-\beta p_{3,1}+\beta p_{2,0}+(1-\beta)p_1+\eta(a_{-2}+a_3)-(2\eta-1)\beta a_{-2},
\end{eqnarray*}
$\beta(p_{3,1}-p_{3,0})+\beta(p_{2,1}-p_{2,0})=(2\eta-1)\beta(a_3-a_{-2})$.
By Lemma \ref{leminv1} (1), $\beta=0$ or 1.

Then
\begin{eqnarray*}
a_3a_{-3}&&=f_{-1}(a_4a_{-2})\\
&&=\beta p_{3,1}+(\beta^2-2\beta)p_{2,0}-(\beta^2-2\beta)p_1+\eta(a_3+a_{-3})+(2\eta-1)(\beta-1)a_{-3}.
\end{eqnarray*}
By $\tau_0$-invariance, $\beta=1$.

Then $p_{2,0}=p_{2,1}$ and $p_{3,1}=p_{3,0}+(2\eta-1)(a_3-a_{-2})$.
Hence $\lambda_1=\lambda_2=\frac{3\eta}{4}$ since $a_0(p_{2,1}-p_{2,0})=0$ and $(p_{2,0}-(\lambda_2-\eta)a_0+\frac{\eta}{2}(a_2+a_{-2}))^2\in\mathbb{F}a_0\oplus M^0(a_0)$.
By considering $a_0(p_{3,1}+p_{3,-1}-2p_{3,0})$., $\eta=\frac{4}{3}$ and 
$$p_{3,0}=\frac{1}{2}p_{2,0}-\frac{1}{2}p_1-\frac{2}{3}(a_3+a_{-3})+\frac{1}{3}(a_2+a_{-2}-a_1-a_{-1}-a_0).$$
Then $2(a_3+a_{-3})-a_2-a_{-2}+a_1+a_{-1}-4a_0\in M^0(a_0)$ and By the fusion rule, $p_1=p_{2,0}$.
Then the fusion rule holds.

\subsection{Proof of Step \ref{stadim8}}

Assume that $\adim=8$.
Then $M$ satisfies an odd relation and $p_{3,1}\neq p_{3,0}$.
So there exist $\alpha,\beta,\gamma\in\mathbb{M}$ auch that
$$a_4-a_{-4}+\alpha(a_3-a_{-3})+\beta(a_2-a_{-2})+\gamma(a_1-a_{-1})=0.$$

Then
\begin{eqnarray*}
a_3a_{-2}=&&f_2(a_1a_{-4})\\
=&&-\alpha p_{4,-1}+(\beta-1)p_{3,0}+(\alpha-\gamma)p_{2,1}+\beta p_1+\eta(a_3+a_{-2})-(2\eta-1)\gamma a_3
\end{eqnarray*}
and
\begin{eqnarray*}
a_3a_{-3}&=&f_1(a_2a_{-4})\\
&=&(\alpha^2-\beta)p_{4,-1}+(\alpha\beta-\alpha-\gamma)p_{3,0}+(-\alpha^2+\alpha\gamma+1)p_{2,1}\\
&&+(-\alpha\beta+\alpha+\gamma)p_1+\eta(a_3+a_{-3})-(2\eta-1)(\alpha\gamma-\beta)a_2.
\end{eqnarray*}
By the invariance of $a_3a_{-3}$, $(2\eta-1)(\alpha\gamma-\beta)(a_4-a_{-4}+a_2-a_{-2})$ and thus $\alpha\gamma-\beta=0$ or $(\alpha,\beta,\gamma)=(0,1,0)$.
But $(\alpha,\beta,\gamma)\neq(0,1,0)$ by considering the structure of $\algenei{2}$.

By the ivariance of $a_3a_{-2}$, 
$$\alpha(p_{4,2}-p_{4,-1})+(\beta-1)(p_{3,1}-p_{3,0})+(\alpha-\gamma)(p_{2,1}-p_{2,0})=(2\eta-1)\gamma(a_3-a_{-2}),$$
$$\alpha(p_{4,1}-p_{4,-1})+(\beta-1)(p_{3,1}-p_{3,-1})=(2\eta-1)\gamma(a_3-a_{-3}+a_2-a_{-2}),$$
$$(\beta-1)(p_{3,1}-p_{2,-1})=(2\eta-1)\gamma((\alpha-1)(a_3-a_{-3})+(\beta-1)(a_2-a_{-2})+(\gamma-1)(a_1-a_{-1}))$$
and
$$0=(2\eta-1)\gamma((\alpha-1)(a_4-a_{-4})+(\alpha+\beta-2)(a_3-a_{-3})+(\alpha+\beta+\gamma-3)(a_2-a_{-2})+(\beta+\gamma-2)(a_1-a_{-1})).$$
Thus $(\alpha,\beta,\gamma)=(\alpha,0,0),(1,1,1)$ or $(2,2,1)$.

If $\beta=\gamma=0$, then $p_{3,1}=p_{3,0}$.

If $\alpha=\beta=\gamma=1$, then $a_4a_{-3}=f_1(a_3a_{-4})=p_{3,1}+\eta(a_4+a_{-3})$.
So $p_{3,1}=p_{3,-1}$.

Thus we may assume $(\alpha,\beta,\gamma)=(2,2,1)$.
Then $a_4a_{-3}=-p_{4,0}+2p_{2,0}+p_1+\eta(a_4+a_{-3})$.
So
$$p_{4,1}-p_{4,0}=\frac{2(2\eta-1)}{3}(a_4-a_{-3}+a_3-a_{-2}+a_2-a_{-1}),$$
$$p_{3,1}-p_{3,0}=(2\eta-1)(a_4+2a_3-a_{-3}+a_2-2a_{-2}-a_{-1})$$
and
$$p_{2,1}-p_{2,0}=\frac{2\eta-1}{3}(a_4-a_{-3}+a_3-a_{-2}+a_2-a_{-1}).$$
However, there did not exist $\eta\in\mathbb{F}$ such that the multiplications $a_0p_{2,1}$ and $a_0(p_{3,1}+p_{3,-1})$ coincide with the resullts of Section 2.

\section*{Acknowledgments}
The author is grateful to Prof.\ Atsushi Matsuo for guidance throughout the work.

\end{document}